\RequirePackage{fix-cm}
\documentclass[smallextended]{svjour3}       
\pdfoutput=1
\smartqed  
\usepackage{graphicx}
\newtheorem{assumption}{Assumption}

\usepackage{amsmath,amsfonts,amssymb} 
\usepackage{algorithmic} 
\usepackage{algorithm}
\usepackage{subfig} 
\usepackage[hyphens]{url} 
\usepackage{tabularx}
\usepackage{booktabs}

\graphicspath{{./}{./figs/}}

  \def\clap#1{\hbox to 0pt{\hss#1\hss}}

\usepackage{bm}
\providecommand{\mat}[1]{\bm{#1}}%
\renewcommand{\vec}[1]{\mathbf{#1}}

\providecommand{\mA}{\ensuremath{\mat{A}}}

\providecommand{\va}{\ensuremath{\vec{a}}}

\providecommand{\vc}{\ensuremath{\vec{c}}}

\providecommand{\vk}{\ensuremath{\vec{k}}}

\providecommand{\vm}{\ensuremath{\vec{m}}}

\providecommand{\vp}{\ensuremath{\vec{p}}}
\providecommand{\vq}{\ensuremath{\vec{q}}}

\providecommand{\vt}{\ensuremath{\vec{t}}}

\providecommand{\vv}{\ensuremath{\vec{v}}}

\providecommand{\vx}{\ensuremath{\vec{x}}}

\providecommand{\vz}{\ensuremath{\vec{z}}}

\newcommand{\ip}[2]{\langle{#1}, {#2}\rangle}

\newcommand{\Prob}[1]{\mathbb{P}\left[#1\right]}
\newcommand{\bmat}[1]{\begin{bmatrix}#1\end{bmatrix}}

\newcommand{\sign}[1]{\operatorname{sign}\left(#1\right)}

\DeclareMathOperator{\R}{\mathbb{R}}
\DeclareMathOperator{\base}{base}
\newcommand{\interior}[1]{\operatorname{int}(#1)}

\newcommand{\conv}[1]{\operatorname{conv}(#1)}
\newcommand{\diam}[1]{\operatorname{diam}(#1)}
\newcommand{\vertices}[1]{\operatorname{vert}(#1)}
\newcommand{\Px}[1]{\mathbb{P}_{\vx}\left[\,#1\,\right]}
\newcommand{\Pv}[1]{\mathbb{P}_{\vv}\left[\,#1\,\right]}

\newtheorem{innercustomthm}{Theorem}
\newenvironment{customthm}[1]
  {\renewcommand\theinnercustomthm{#1}\innercustomthm}
  {\endinnercustomthm}

\newtheorem{innercustomcor}{Corollary}
\newenvironment{customcor}[1]
  {\renewcommand\theinnercustomcor{#1}\innercustomcor}
  {\endinnercustomcor}

\begin{document}

\title{A randomized algorithm for enumerating zonotope vertices
}


\author{Kerrek Stinson\and
David F.~Gleich\and
Paul G.~Constantine
}


\institute{K. Stinson \at
Colorado School of Mines, Golden, CO 80401\\
\email{kstinson@mines.edu}
\and
D.F. Gleich \at
Purdue University, West Lafayette, IN 47907\\
\email{dgleich@purdue.edu}
\and
P.G. Constantine \at
Colorado School of Mines, Golden, CO 80401\\
\email{paul.constantine@mines.edu}
}

\date{Received: date / Accepted: date}

\maketitle

\begin{abstract}
We propose a randomized algorithm for enumerating the vertices of a zonotope, which is a low-dimensional linear projection of a hypercube. 
The algorithm produces a pair of the zonotope's vertices by sampling a random linear combination of the zonotope generators, where the combination's weights are the signs of the product between the zonotope's generator matrix and random vectors with normally distributed entries. 
We study the probability of recovering particular vertices and relate it to the vertices' normal cones. This study shows that if we terminate the randomized algorithm before all vertices are recovered, then the convex hull of the resulting vertex set approximates the zonotope. In high dimensions, we expect the enumeration algorithm to be most appropriate as an approximation algorithm---particularly for cases when existing methods are not practical.

\end{abstract}

\section{Introduction}
\label{sec:intro}

A \emph{zonotope} is a convex, centrally symmetric polytope that is the $n$-dimensional linear projection of an $m$-dimensional hypercube. More precisely, let $\mA\in\mathbb{R}^{n\times m}$ with $n<m$, and define the zonotope $Z\subset\mathbb{R}^n$ as
\begin{equation}
\label{eq:zono}
Z \;=\; Z(\mA) \;=\; \left\{\,
\mA\,\vx \,\mid\, \vx \in [-1,1]^m
\,\right\}.
\end{equation}
Figure \ref{fig:cube} shows a zonotope with $m=3$ and $n=2$. The zonotope has an equivalent representation as a Minkowski sum of $m$ line segments in $\mathbb{R}^n$. Let $\va_i\in\mathbb{R}^n$ be the $i$th column of $\mA$. Recall that the Minkowski sum $P+Q$ of sets $P$ and $Q$ is $P+Q = \{\,\vp+\vq\,\mid\,\vp\in P,\,\vq\in Q\,\}$. Then
\begin{equation}
\label{eq:zonomink}
Z \;=\; A_1 + \cdots + A_m,
\end{equation}
where $A_i\subset\mathbb{R}^n$ is defined as
\begin{equation}
\label{eq:lineseg}
A_i \;=\; \{\,\gamma\,\va_i \,\mid\, \gamma\in[-1,1] \,\}, 
\qquad i=1,\dots,m.
\end{equation}
The vectors $\{\va_i\}$ are called the \emph{generators} of $Z$.

\begin{figure}[t!]
\centering
\includegraphics[width=0.35\linewidth]{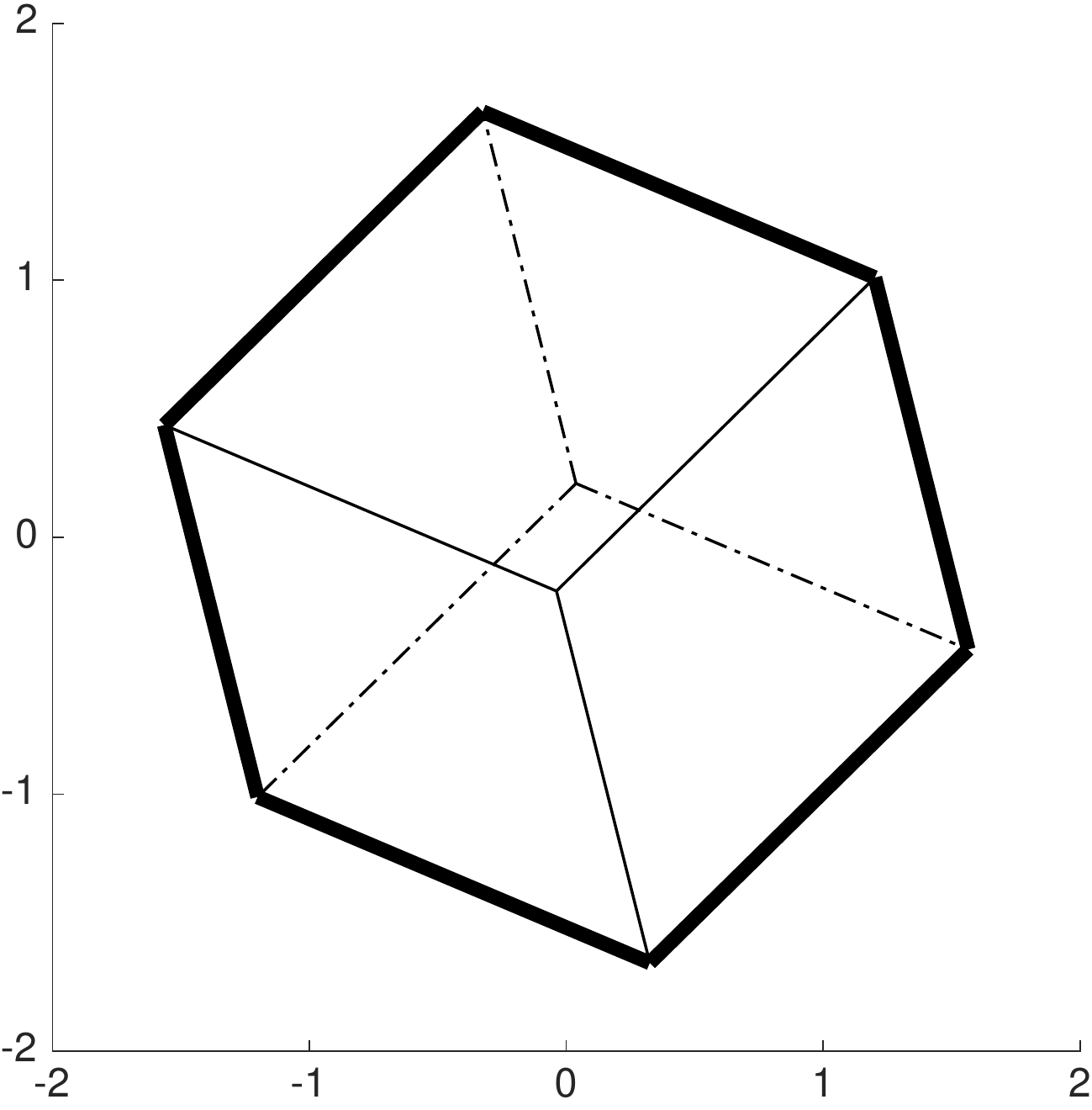}
\caption{A three-dimensional cube $[-1,1]^3$ rotated and photographed. The dotted lines show the cube's edges in the background. The thick lines show the boundary of the two-dimensional zonotope.}
\label{fig:cube}
\end{figure}

\subsection{Applications of zonotopes}
Zonotopes serve as bounding volumes in collision detection~\cite{Guibas2003} and aid fault diagnosis in control systems~\cite{Scott2014}. They are also helpful for understanding solutions to underdetermined linear systems found in compressed sensing~\cite{Donoho2010}. 
Other applications require explicit enumeration of the zonotope vertices, including the fixed rank integer quadratic program~\cite{Ferrez2005}, $L_1$-subspace signal processing~\cite{Markopoulos2014}, and linear regression models with interval data~\cite{Cerny2013}. 

Our particular interest in zonotope construction arises in the context of approximating functions of several variables with \emph{generalized ridge functions}~\cite{pinkus2015}. Define the function $f:[-1,1]^m\rightarrow\mathbb{R}$, and consider an approximation of the form
\begin{equation}
\label{eq:ridge}
f(\vx) \;\approx\; g(\mA\,\vx),\qquad \vx\in[-1,1]^m,
\end{equation}
where $\mA$ has orthogonal rows, and $g(\cdot)$ is a function from $\mathbb{R}^n$ to $\mathbb{R}$. The rows of $\mA$ are the first $n$ eigenvectors of a symmetric positive semidefinite matrix derived from $f$'s gradient; they define $f$'s \emph{active subspace}~\cite{asmbook}. To exploit the approximation \eqref{eq:ridge}, one must explicitly represent $g$'s domain, which is a zonotope.

\subsection{Algorithms for vertex enumeration}

The zonotope's vertices are a subset of the corners of the hypercube---i.e., $m$-vectors with elements in $\{-1,1\}$---projected with the matrix $\mA$. A straightforward algorithm for zonotope vertex enumeration is to use convex hull algorithm such as Quickhull~\cite{Barber1996} to compute the convex hull of the set $\left\{\mA\,\vv\,\mid\,\vv\in\{-1,1\}^m\right\}$. However, the complexity of this approach scales exponentially in $m$, since there are $2^m$ vertices of the hypercube, so this is not practical for large $m$. 

An alternative approach developed by Ferrez, et al.~\cite{Ferrez2005} based on \emph{reverse search}~\cite{Avis1996} has time complexity $\mathcal{O}(m\,n\,\text{LP}(m,n)\,|V|)$, where (i) $\text{LP}(m,n)$ is the complexity of solving a linear program with $m$ inequalities in $n$ variables, (ii) $V$ is the set of zonotope vertices, and (iii) $|V|$ is the number of zonotope vertices. This algorithm admits an efficient parallel implementation~\cite{Weibel2010}, and open source software is available~\cite{libzonotope}. An alternative to explicit construction is to approximate the zonotope with an ellipse via Goffin's algorithm~\cite{Cerny2012}. However, numerical implementation of this approach appears problematic, and we know of no existing implementations. 

\subsection{A randomized algorithm for enumeration}

We propose and analyze a randomized algorithm for constructing a zonotope $Z$ via vertex enumeration given its generators in the form of an $n\times m$ matrix $\mA$. Algorithm \ref{alg:vertenum} contains the simplest form of the algorithm, which is appropriate under some mild conditions on $\mA$. The algorithm relies on characterizing a zonotope vertex as $\mA\sign{\mA^T\vx}$ for $\vx\in\mathbb{R}^n$, where $\sign{\vv}$ returns a vector of elements in $\{-1,1\}$ that correspond to signs of $\vv$'s components\footnote{We assume that $\mA$ is such that entries in $\vv$ are zero with probability zero.}; we derive this characterization in Section \ref{sec:zonotopes}.

\begin{algorithm}[h!]
\caption{A randomized algorithm to enumerate vertices of the zonotope $Z=Z(\mA)$ given generators $\mA$.}
\label{alg:vertenum}
\begin{algorithmic}
\STATE{Let $k$ be the number of $Z$'s vertices.}
\STATE{Initialize the empty set of vertices $V=\varnothing$.}
\WHILE{$|V| < k$}
\STATE{Draw $\vx$ independently from a standard Gaussian distribution on $\mathbb{R}^n$.}
\STATE{Compute $\vv_+=\mA\sign{\mA^T\vx}$ and $\vv_-=-\vv_+$.}
\IF{$V$ does not contain $\vv_+$}
\STATE{Add $\vv_+$ and $\vv_-$ to $V$.}
\ENDIF
\ENDWHILE
\end{algorithmic}
\end{algorithm}

If Algorithm \ref{alg:vertenum} is terminated before $|V|=k$, the convex hull of the vertices in $V$ is contained within $Z$. 
In Section \ref{sec:alg}, we derive a lower bound on the number of random samples needed such that the Hausdorff distance between $Z$ and the convex hull of $V$ is below a user-specified with high probability. Section \ref{sec:exp} presents several illustrative numerical experiments.

\subsection{Notation}

We use the following notation conventions throughout the paper. Uppercase letters (e.g., $A$) are sets; uppercase boldface letters (e.g., $\mA$) are matrices; and lowercase boldface letters (e.g., $\va$) are vectors. For $A \subset \R^n$, $\interior{A}$, $\conv{A}$, and $\diam{A}$ are the interior,  convex hull, and diameter of $A$, respectively. For a polytope $P$, $\vertices{P}$ returns the set of $P$'s vertices. Two vertices of a polytope are adjacent if there is an edge connecting them. The inner product between vectors $\vp$ and $\vq$ is $\ip{\vp}{\vq}$. 

\section{Characterizing a zonotope vertex}
\label{sec:zonotopes}


The statements in this section and the next section utilize the following set of assumptions.
\begin{assumption} \label{ass:big}
 The $n\times m$ matrix $\mA=\bmat{\va_1&\cdots &\va_m}$ generates the zonotope $Z$ as in \eqref{eq:zono} and \eqref{eq:zonomink} with $m \ge n$.
 We assume that (i) $\mA$ has no columns of all zeros and (ii) no two columns of $\mA$ are scalar multiples of each other. 
\end{assumption}
Together, these assumptions imply that the zonotope is in general position.


Before introducing the main result of this section, we recall a theorem from Ferrez, et al.~\cite{Ferrez2005} that counts the number of vertices for a zonotope in general position. 

\begin{theorem}[Theorem 3.1~\cite{Ferrez2005}]
\label{thm:numvert}
The number $|\vertices{Z(\mA)}|$ of vertices of the zonotope $Z(\mA)$ satisfies
\begin{equation}
\label{eq:numvert}
|\vertices{Z(\mA)}| \;\leq\; 2\sum_{i=0}^{n-1} {m-1 \choose i},
\end{equation}
where equality is attained if $Z(\mA)$ is in general position.
\end{theorem}
The calculation in \eqref{eq:numvert} produces $k$ in Algorithm \ref{alg:vertenum}. 

The following theorem and two corollaries are the main results of this section. All proofs are found in Appendix \ref{app:0}; the proof of Theorem \ref{thm:zono} leverages a result from Fukuda~\cite[Corollary 2.2]{Fukuda2004}.

\begin{theorem} 
\label{thm:zono} 
Under Assumption~\ref{ass:big}, for $\vx \in \R^n$ such that $\mA^T \vx$ has all nonzero components, the point $\vv$ defined as
\begin{equation}
\label{eq:vert}
\vv \;=\; \vm(\vx) \;:=\; \mA\sign{\mA^T\vx}
\end{equation}
is a vertex of the zonotope $Z(\mA)$. 
\end{theorem}

\begin{corollary}
\label{cor:vert-}
If $\vv$ is a vertex of the zonotope $Z(\mA)$, then so is $-\vv$.
\end{corollary}

\begin{corollary} 
\label{cor:mapping}
Under Assumption~\ref{ass:big} and for $\vm(\vx)$ as in \eqref{eq:vert}, define $H\subset\R^n$ as
\begin{equation}
\label{eq:H}
H \;=\; \bigcup_{i=1}^m \{\, \vx\in \R^n\,\mid\, \ip{\va_i}{\vx} = 0 \,\}.
\end{equation}
The mapping $\vm:\R^n \setminus H\to \vertices{Z(\mA)}$ is well defined and onto.
\end{corollary}



\section{A randomized algorithm for approximating a zonotope}
\label{sec:alg}

The zonotope vertex characterization in Theorem \ref{thm:zono} underpins the randomized algorithm in Algorithm \ref{alg:vertenum}. Next we study the algorithm's approximation properties. Namely, if Algorithm \ref{alg:vertenum} is terminated before $|V|=k$, then the convex hull of the set $V$ approximates the zonotope. We seek a lower bound on the number of random samples ($\vx$ in Algorithm \ref{alg:vertenum}) needed such that the Hausdorff distance between $Z$ and the convex hull of $V$ is below a user-specified tolerance with high probability.

\subsection{A probability measure on the zonotope vertices}
Let $\mathbb{P}_{\vx}$ be a probability measure on $\R^n$. Recall the definition of the set $H$ in \eqref{eq:H}, which is the set of vectors in $\R^n$ that are orthogonal to at least one of the generators $\{\va_i\}$. Note that $\Px{H}=0$ because the null-space of a column of $\mA$ is a subspace of dimension $n-1$. 
Therefore, by Corollary \ref{cor:mapping},
\begin{equation}
\Px{
\{\,\vx\in\R^n \,\mid\, \mA\sign{\mA^T\vx}\in\vertices{Z(\mA)} \,\}
} \;=\; 1.
\end{equation}
In other words, the probability that $\vx\in\R^n$ maps to some zonotope vertex is one. We now ask what the probability is that $\vx$ maps to a specific vertex. We show that this probability is characterized by a geometric feature of the vertex, namely its associated normal cone. Loosely speaking, for an appropriate $\mathbb{P}_{\vx}$ such as a standard Gaussian measure, a vertex gets mapped to more frequently if the zonotope is sharper at that vertex---and such vertices influence the convex hull more. If the zonotope is nearly flat around a vertex, then this vertex does not contribute much to the zonotope's definition; we show that the probability of mapping to these non-influential vertices is relatively low. Such insight offers hope for Algorithm \ref{alg:vertenum}'s viability as an approximation algorithm for high-dimensional cases when alternatives are impractical. 

A similar idea arises in recent work on a randomized algorithm for archetypal analysis of large data sets~\cite{damle2014random}, where the goal is to identify \emph{archetypes}---which are elements of the data set---such that the remaining points can be represented as convex combinations of the archetypes. We borrow heavily from their geometric analysis. 

In what follows, let $\vv$ be a vertex of the zonotope $Z=Z(\mA)$. Recall the \emph{normal cone} of $\vv$, denoted $N_Z(\vv)$, is 
\begin{equation} \label{eqn:ncone}
N_Z(\vv) \;=\; \{\,
\vx\in\R^n \,\mid\, \ip{\vx}{\vz-\vv}\leq 0 
\mbox{ for all $\vz\in Z$} 
\,\}.
\end{equation}
Figure \ref{fig:ncone} shows an example of a normal cone for a two-dimensional polytope's vertex. 

\begin{figure}[t!]
\centering
\includegraphics[width=0.55\linewidth]{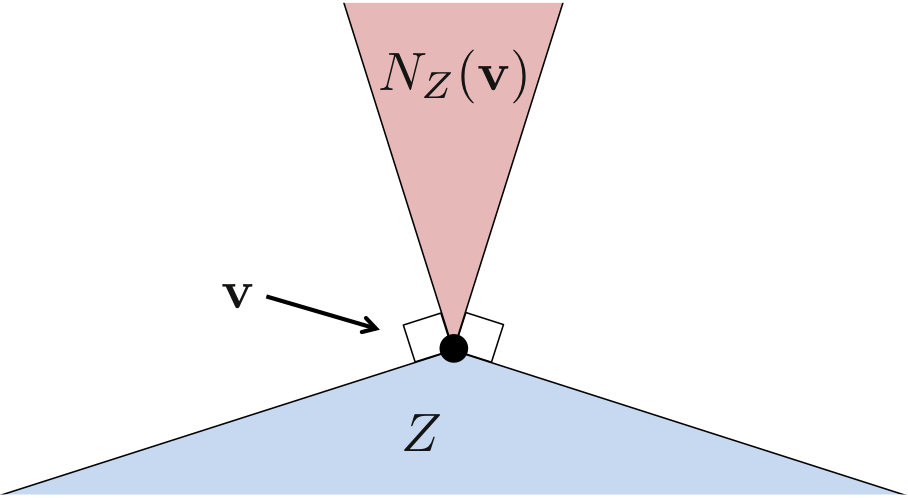}
\caption{The red shaded region shows an example of a normal cone $N_Z(\vv)$ for a vertex $\vv$ of the blue shaded polytope $Z$.}
\label{fig:ncone}
\end{figure}



\begin{theorem} 
\label{thm:normalCone}
Under Assumption~\ref{ass:big}, for $\vv\in\vertices{Z}$ and the mapping $\vm(\vx)$ from \eqref{eq:vert}, 
\begin{equation}
\vm^{-1}(\vv) \;=\; \interior{N_Z(\vv)}.
\end{equation}
\end{theorem}
The proof of Theorem \ref{thm:normalCone} is in Appendix \ref{app:A}. This theorem enables us to construct a probability measure $\mathbb{P}_{\vv}$ on the vertices of the zonotope $Z$. For $\vv_i\in\vertices{Z}$, define
\begin{equation}
\label{eq:probv}
\begin{aligned}
\Pv{\vv_i} 
&:= \Px{N_Z(\vv_i)} \\ 
&= \Px{\interior{N_Z(\vv_i)}} \\
&= \Px{
\{\,\vx\in\R^n \,\mid\, \mA\sign{\mA^T\vx}=\vv_i \,\}
}.
\end{aligned}
\end{equation}
The first equality follows since  $N_{Z}(\vv_i)\setminus \interior{N_{Z}(\vv_i)}$ is a subset of finitely many hyperplanes. 

\subsection{Number of samples for zonotope approximation}
Algorithm \ref{alg:vertenum} uses a standard Gaussian measure for $\mathbb{P}_{\vx}$. In fact, any centrally symmetric distribution should suffice. However, with a Gaussian measure, we can apply work by Damle and Sun~\cite{damle2014random} to derive a lower bound on the number of samples needed to all vertices with sufficiently large \emph{simplicial constants}. A vertex's simplicial constant quantifies how far the vertex is from the convex hull of the remaining vertices. More precisely, the simplicial constant $\alpha_Z$ for a vertex $\vv$ of $Z$ is
\begin{equation}
\label{eq:simpconst}
\alpha_Z(\vv) \;=\; \inf_{\vx} 
\{\, 
\|\vv-\vx\|_2 \,\mid\, 
\vx \in \conv{\vertices{Z}\setminus \{\vv\}}
\,\}.
\end{equation}
Note that $\alpha_Z(\vv)$ is precisely the Hausdorff distance between $Z$ and $\conv{\vertices{Z} \setminus \{\vv\}}$. Thus, if the simplicial constant of $\vv$ is small, then $Z$ is well approximated by $\conv{\vertices{Z} \setminus \{\vv\}}$. Also note that $\alpha_Z(\vv)=\alpha_Z(-\vv)$ by the central symmetry of the zonotope. 

Damle and Sun~\cite{damle2014random} derive an inequality that relates a vertex's simplicial constant to the Gaussian measure of its normal cone. To apply this result, we define the \emph{base} of a vertex $\vv\in\vertices{Z}$ as
\begin{equation}
\base_Z(\vv) \;=\; 
\conv{
\{\, 
\vx \in \vertices{Z}\setminus\{\vv\} \,\mid\, 
\mbox{ $\vx$ is adjacent to $\vv$ in $Z$} 
\,\}
}
\end{equation}
For example, if $Z\subset\R^2$, then $\base_Z(\vv)$ is the edge nearest $\vv$ in $\conv{\vertices{Z}\setminus\{\vv\}}$. 

\begin{theorem} 
\label{thm:bounds}
Given $\varepsilon>0$ and $\delta>0$, let $K$ be the convex hull of a centrally symmetric set of points, and let $\vx_1,\dots,\vx_p$ be independent standard Gaussian vectors in $\mathbb{R}^n$. Define the subset $U_{K}\subseteq\vertices{K}$ dependent on $\delta$ as
\begin{equation}
U_{ K} \;=\; \{\, \vv\in\vertices{K} \,\mid\, \alpha_K(\vv)\geq \delta \,\},
\end{equation}
and define the event $A_K$ dependent on $\{\vx_i\}$ as
\begin{equation}
\label{eq:rv}
\{\vx_i\} \cap (N_K(\vv) \cup N_K(-\vv)) \not= \varnothing \mbox{ for all $\vv\in U_K$}.
\end{equation}
If $p$ is such that
\begin{equation}
p \;>\; \frac{
\log(|\vertices{K}|/\varepsilon)
}{
\log(1/(1-k))
},
\end{equation}
where
\begin{equation}
k \;=\; \left(
\frac{1}{2}\,(1-\sin(\arctan(b/\delta)))
\right)^{\frac{n-1}{2}},
\end{equation}
and
\begin{equation}
b \;\geq\; \underset{\vv\in\vertices{K}}{\max}
\diam{\base_K(\vv)},
\end{equation}
then $\mathbb{P}[A_K] \geq 1 - \varepsilon$.
\end{theorem}
The proof of Theorem \ref{thm:bounds} is in Appendix \ref{app:B}. The random variable $T$ equals 1 when the Gaussian samples $\{\vx_i\}$ fall in the normal cones of all vertices (or their negatives) with simplicial constant greater than $\delta$. Theorem \ref{thm:bounds} may be applied with $K = Z$; then by Theorem \ref{thm:normalCone}, $T=1$ means Algorithm \ref{alg:vertenum} returns all vertices with simplicial constant greater than $\delta$.

In the application to active subspaces (see \eqref{eq:ridge}), the rows of $\mA$ are orthogonal. In this case, we obtain a bound that is independent of any unknown property of the zonotope, which we state as a corollary without proof.
\begin{corollary}
Suppose the rows of $\mA$ are orthogonal. Then 
\begin{equation}
\underset{\vv\in\vertices{Z}}{\max} \diam{\base_Z(\vv)}
\;\leq\; 
\underset{\vx\in[-1,1]^m}{\max} 2\|\mA\,\vx\|_2
\;\leq\; 
2\sqrt{m}
\end{equation}
and Theorem \ref{thm:bounds} holds with $K=Z(\mA)$ and $b = 2\sqrt{m}$.
\end{corollary}

We conclude this section with a bound on the Hausdorff distance between $Z$ and the convex hull of the set $V$ produced by Algorithm \ref{alg:vertenum}. We denote the Hausdorff distance by $h(Z,\conv{V})$. 

\begin{theorem}
\label{thm:hdist}
Let $Z=Z(\mA)$ be a zonotope with generator $\mA\in\mathbb{R}^{n\times m}$ satisfying Assumption \ref{ass:big}. Given $\varepsilon>0$ and $\delta>0$, choose $p$ as in Theorem \ref{thm:bounds} for $b\geq\operatorname{diam}(Z)$, and let $V$ be the subset of $Z$'s vertices produced by Algorithm \ref{alg:vertenum} after $p$ iterations. Then
\begin{equation}
h(Z,\conv{V}) \;\leq\; \frac{|\vertices{Z}\setminus V|}{2}\, \delta
\end{equation} 
with probability at least $1-2^a\,\epsilon$, where $a=|\vertices{Z}\setminus U_{Z}|/2$ and 
\begin{equation}
U_{Z} \;=\; \{\,\vv \in \vertices{Z} \,\mid\, \alpha_Z(\vv) \geq \delta \,\}.
\end{equation}
\end{theorem}
The proof of Theorem \ref{thm:hdist} is in Appendix \ref{app:C}. The result quantifies the approximation of $\conv{V}$ to $Z$.

\section{Numerical experiments}
\label{sec:exp}

The following experiments illustrate and complement the theory developed in previous sections. The Python scripts for running the experiments and generating the figures (excluding Figure \ref{fig:res0}, which was generated in Matlab) are available at \url{https://bitbucket.org/paulcon/a-randomized-algorithm-for-enumerating-zonotope-vertices}.

Figure \ref{fig:res0} shows a zonotope with $m=5$, $n=2$, where the generating matrix $\mA\in\R^{2\times 5}$ is randomly generated with orthogonal rows. The black circles in Figure \ref{fig:z0} are the projections of all $2^5$ corners of the $5$-dimensional hypercube. The red x's identify the subset of projected vertices that define the zonotope---computed with Quickhull---which is outlined in black and shaded in gray. Figure \ref{fig:z1} visualizes the map $\vm(\vx)$ from \eqref{eq:vert}. The zonotope is outlined in black, and the vertices are black circles with particular face colors. The colors in the plane indicate to which vertex the corresponding points in $\R^2$ map to. This illustrates the probability measure on the vertices and its relationship to the simplicial constant from \eqref{eq:simpconst}. For example, the relatively large yellow regions map to the yellow vertices, which have a large simplicial constant---as indicated by the zonotope's sharpness around the vertex. In contrast, the relatively small dark blue regions map to the dark blue vertices with very small simplicial constants. The convex hull of all vertices excluding the dark blue vertices is a good approximation of the zonotope. This illustrates the theory behind Theorem \ref{thm:bounds}. 

\begin{figure}[!ht]
\centerline{
\subfloat[]{
\includegraphics[width=0.45\textwidth]{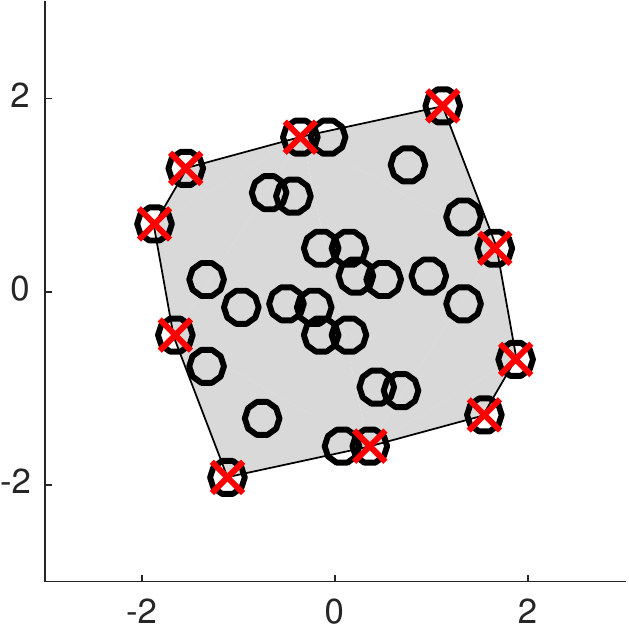}%
\label{fig:z0}
}
\hfil
\subfloat[]{
\includegraphics[width=0.45\textwidth]{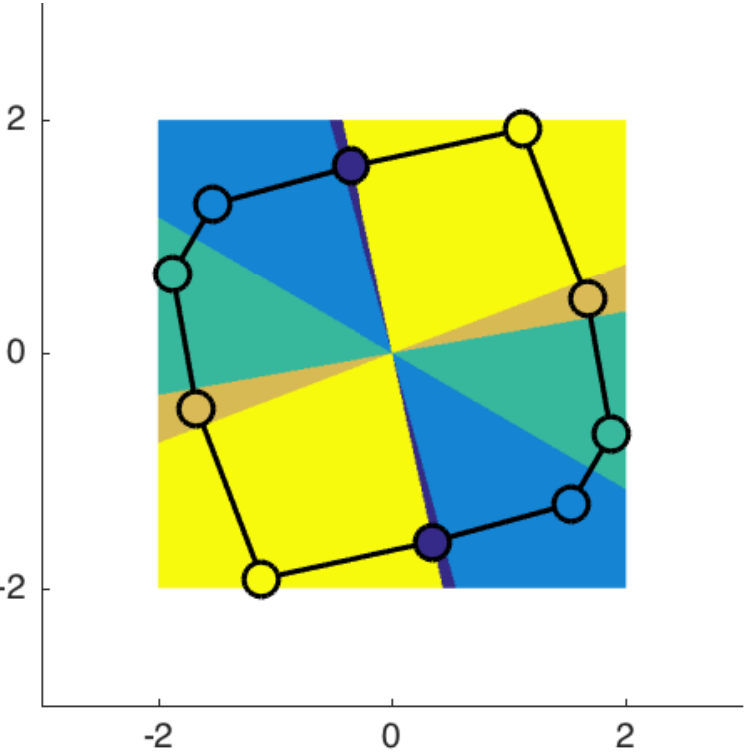}%
\label{fig:z1}
}
}
\caption{Visualizing a zonotope with $m=5$ and $n=2$. Figure \ref{fig:z0} shows all $2^5$ vertices of the $5$-dimensional hypercube projected with the generator matrix. The red x's identify the subset of projected vertices that define the zonotope. Figure \ref{fig:z1} illustrates the probability map on the zonotope vertices and confirms the theory behind Theorem \ref{thm:bounds}.}
\label{fig:res0}
\end{figure}

Figure \ref{fig:errors0} shows the results of an approximation experiment. For $m=10$ and $n=2,\dots,5$, we generate a random orthogonal generating matrix. We then run 10 independent trials of Algorithm \ref{alg:vertenum} and check the approximation error after a fixed number $p$ of random samples, which we compute as
\begin{equation}
\text{error}_p \;=\; h(Z,\,\conv{V_p}),
\end{equation}
where $V_p$ is the resulting vertex set after $p$ samples. The subfigures in Figure \ref{fig:errors0} show the decrease in error for each trial as a function of $p$ for the different values of $n$; a missing value indicates that the error is zero. The convergence rate degrades slightly as $n$ increases. Figure \ref{fig:errors1} shows results of an identical study with $m=20$. We observe similarly degrading convergence rate as $n$ increases. Surprisingly, the results are similar between $m=10$ and $m=20$. However, we note that for $m=20$ and $n=4$, Algorithm \ref{alg:vertenum} found only 2310 of the 2320 zonotope vertices after $10^{10}$ independent samples, which we used as the \emph{true} $Z$ for the experiment. Similarly, for $m=20$ and $n=5$, Algorithm \ref{alg:vertenum} found 9760 of the 10072 vertices after $10^{10}$ independent samples. We repeated these studies with different realizations of the generating matrix with comparable results. 

\begin{figure}[!ht]
\centering
\subfloat[$n=2$]{
\includegraphics[width=0.45\textwidth]{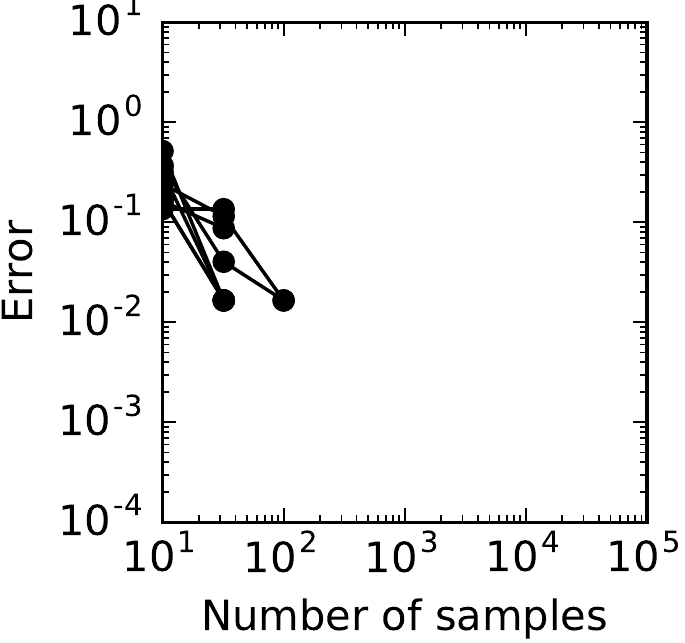}%
}
\hfil
\subfloat[$n=3$]{
\includegraphics[width=0.45\textwidth]{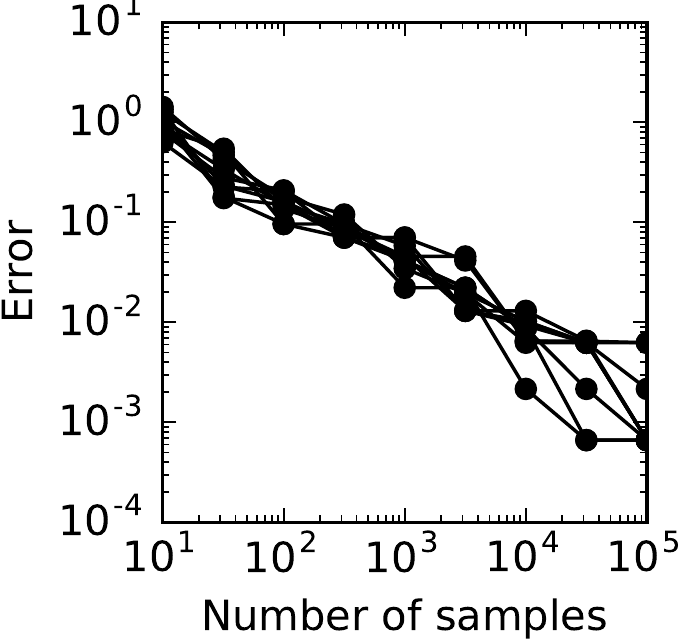}%
}\\
\subfloat[$n=4$]{
\includegraphics[width=0.45\textwidth]{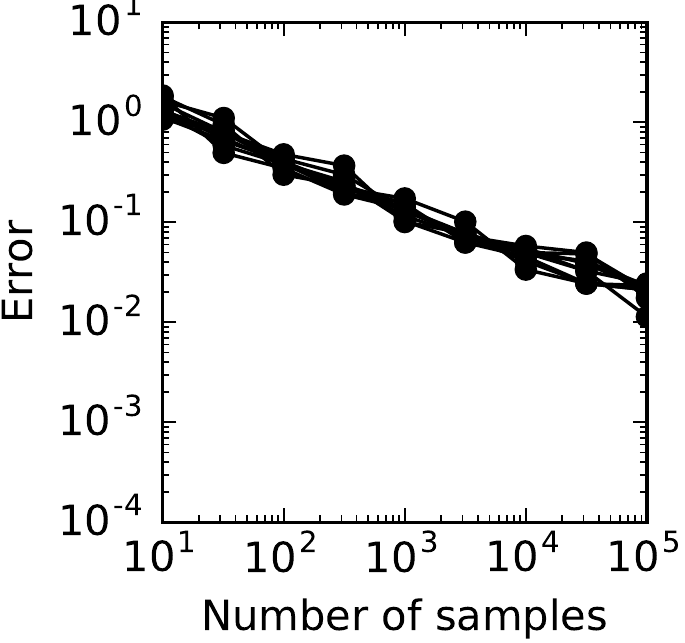}%
}
\hfil
\subfloat[$n=5$]{
\includegraphics[width=0.45\textwidth]{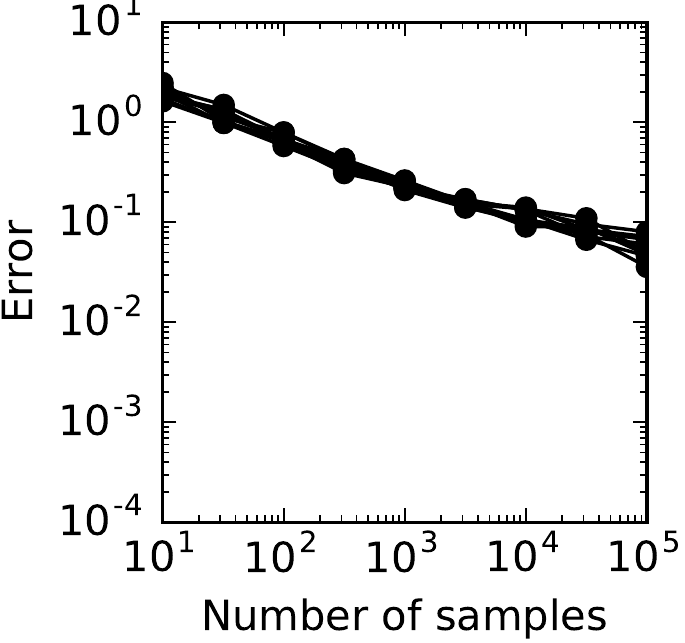}%
}
\caption{Errors in the approximate zonotope from Algorithm \ref{alg:vertenum} after $p$ samples for $m=10$ and $n$ increasing from 2 to 5. The convergence rate degrades as $n$ increases.}
\label{fig:errors0}
\end{figure}

\begin{figure}[!ht]
\centering
\subfloat[$n=2$]{
\includegraphics[width=0.45\textwidth]{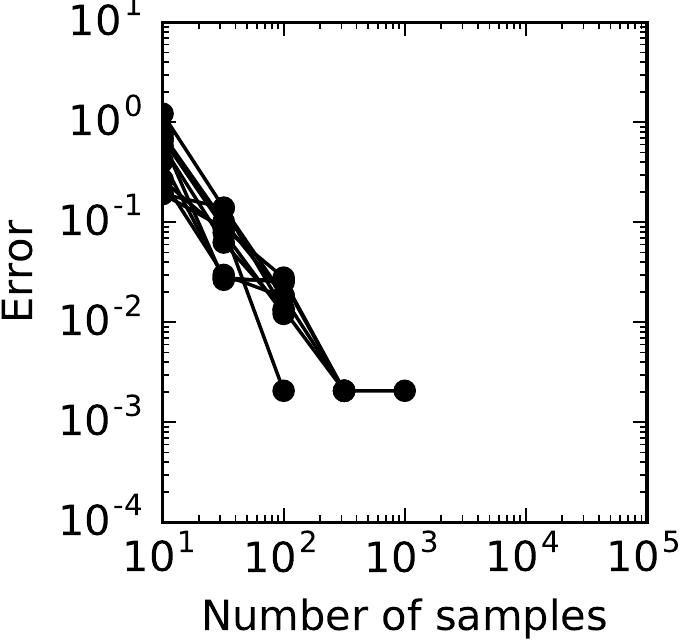}%
}
\hfil
\subfloat[$n=3$]{
\includegraphics[width=0.45\textwidth]{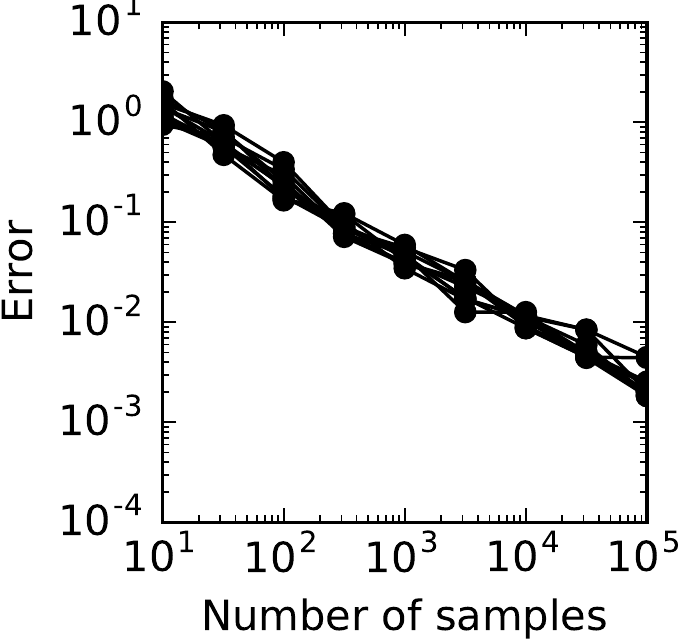}%
}\\
\subfloat[$n=4$]{
\includegraphics[width=0.45\textwidth]{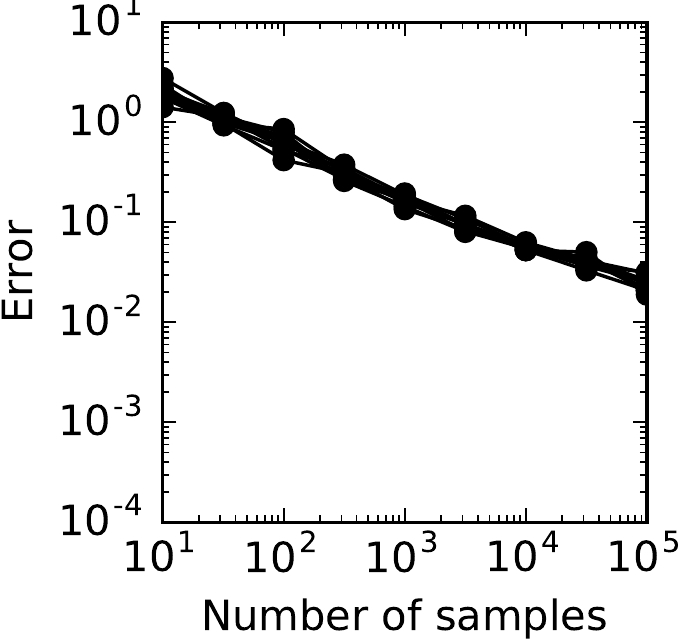}%
}
\hfil
\subfloat[$n=5$]{
\includegraphics[width=0.45\textwidth]{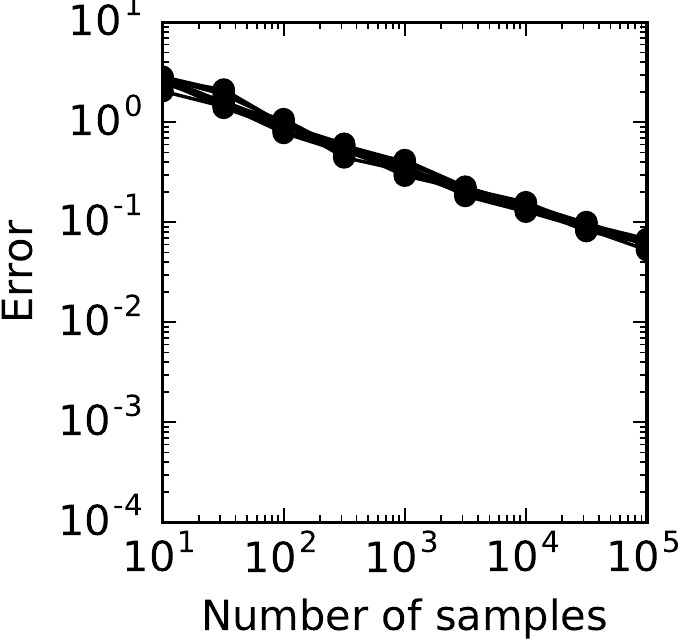}%
}
\caption{Errors in the approximate zonotope from Algorithm \ref{alg:vertenum} after $p$ samples for $m=20$ and $n$ increasing from 2 to 5. Similar to the $m=10$ case in Figure \ref{fig:errors0}, the convergence rate degrades as $n$ increases. However, the convergence rate for $m=20$ appears similar to the convergence rate for $m=10$.}
\label{fig:errors1}
\end{figure}

Figure \ref{fig:stops} shows histograms of the number of samples needed in Algorithm \ref{alg:vertenum} for complete vertex enumeration for $10^4$ independent trials using the same generator matrices. Note the general increase in the number of samples needed for complete enumeration as $m$ and $n$ increase. This is most likely a result of the random method for generating the generator matrices. As $m$ and $n$ increase, the number of vertices defining the zonotope increases, and the chances are good of getting a vertex with a very small simplicial constant. This factor controls the randomized method's ability to enumerate all the zonotope's vertices. 

\begin{figure}[!ht]
\centering
\subfloat[$m=10$, $n=2$]{
\includegraphics[width=0.45\textwidth]{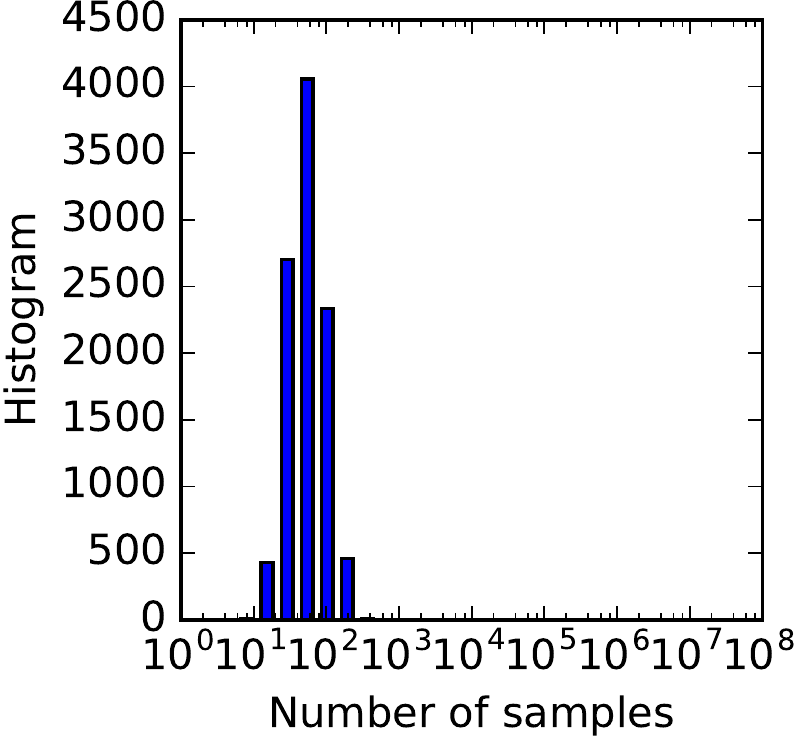}%
}
\hfil
\subfloat[$m=10$, $n=3$]{
\includegraphics[width=0.45\textwidth]{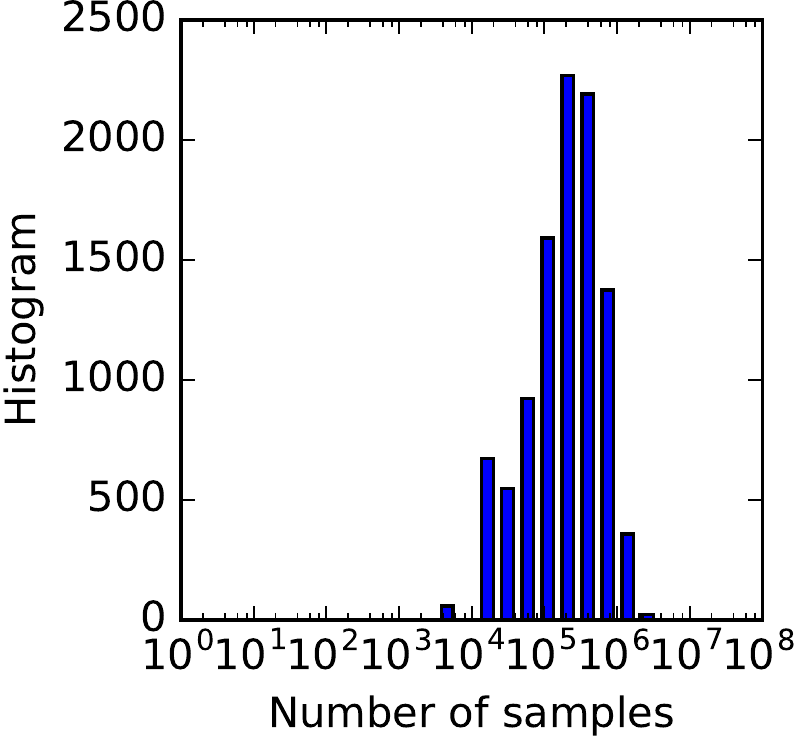}%
}\\
\subfloat[$m=20$, $n=2$]{
\includegraphics[width=0.45\textwidth]{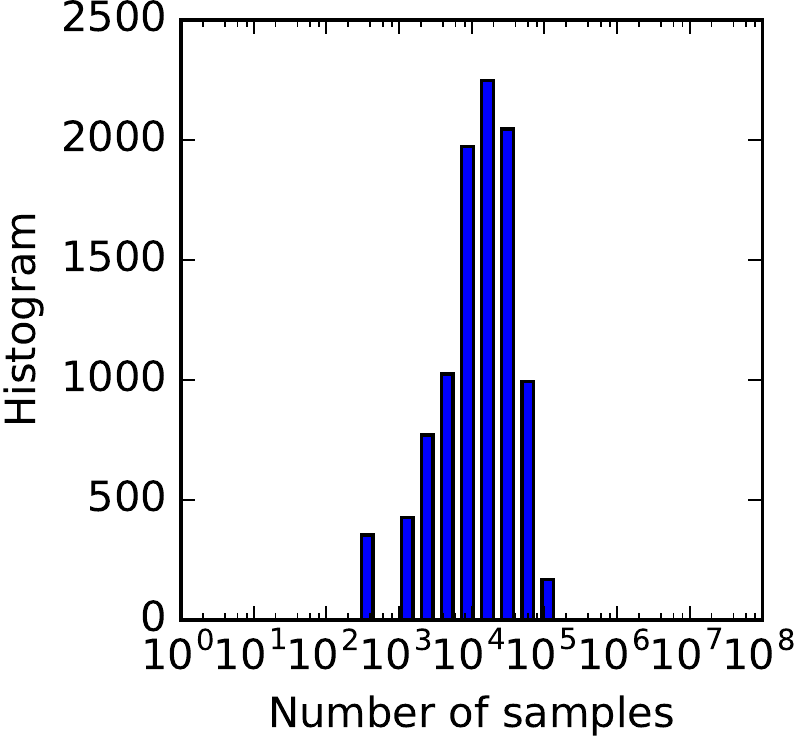}%
}
\hfil
\subfloat[$m=20$, $n=3$]{
\includegraphics[width=0.45\textwidth]{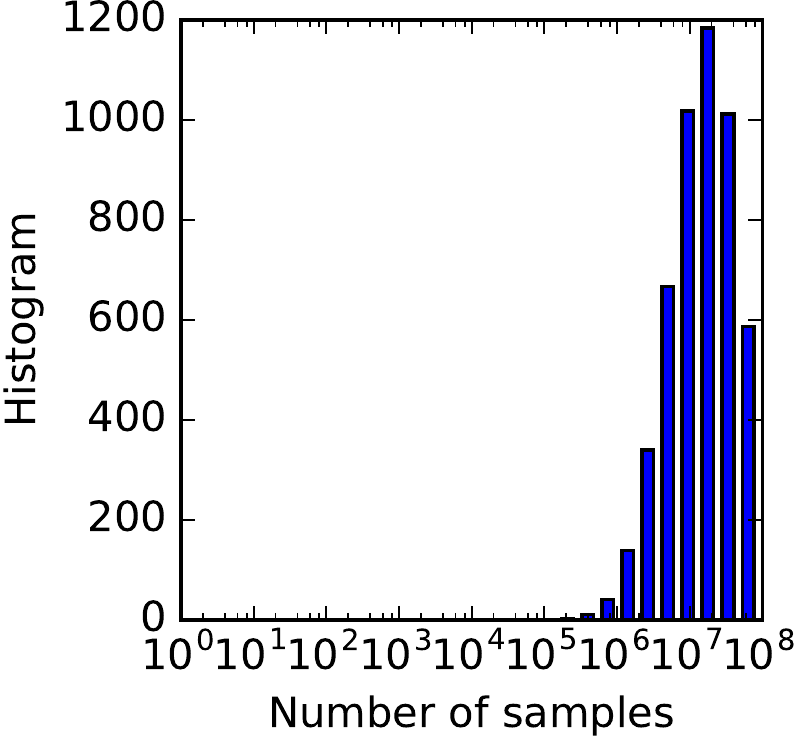}%
}
\caption{Histograms of number of samples needed for complete zonotope vertex enumeration for varying values of $m$ and $n$ over $10^4$ independent trials of Algorithm \ref{alg:vertenum}.}
\label{fig:stops}
\end{figure}

Finally, we compare wall clock times between a Python/numpy implementation of Algorithm \ref{alg:vertenum} with complete enumeration (i.e., terminate when $|V|=k$) to the Python-wrapped \texttt{libzonotope} implementation of the reverse search algorithm~\cite{libzonotope} for various values of $m$ and $n$ using randomly generated generator matrices. Table \ref{tab:times} shows the averages and standard deviations of the timings in seconds over 1000 trials. The times for the randomized algorithm are labeled \emph{random}, and the times for the reverse search algorithm are labeled \emph{reverse}. For $n=2$, the randomized algorithm appears to be faster, with the advantage increasing as $m$ increases. However, for $n=3$ the reverse search implementation is much faster---especially for larger $m$. It's worth noting the large standard deviation for the randomized algorithm when $n=3$ and $m$ is large.

\begin{table}
\caption{Timings in seconds for a Python/numpy implementation of Algorithm \ref{alg:vertenum} (random) and the Python-wrapped \texttt{libzonotope} implementation of a reverse search enumeration (reverse). The faster average time is bolded.}
\label{tab:times}
\begin{tabularx}{\linewidth}{llXXXXXXXX}
	\toprule
& & $n = 2$ & & & & $n = 3$ \\
  \cmidrule(r){3-6} \cmidrule(r){7-10}
  & $m =$ & 5 & 10 & 15 & 20 & 5 & 10 & 15 & 20 \\
	\midrule
avg. & random &	0.067	&	\textbf{0.095}	&	\textbf{0.13}	&	\textbf{0.145}	&	\textbf{0.072}	&	1.524	&	38.64	&	 39.38 \\
 & reverse &	\textbf{0.046}	&	0.211	&	0.388	&	0.589	&	{0.201}	&	\textbf{0.95}	&	\textbf{1.814}	&	 \textbf{3.700} \\ \addlinespace
std. & random &	0.021	&	0.024	&	0.043	&	0.035	&	0.025	&	1.598	&	28.47	&	 33.81 \\
&  reverse & 	0.016	&	0.063	&	0.152	&	0.287	&	0.056	&	0.367	&	0.376	&	 0.384\\ 
	\bottomrule
\end{tabularx}
\end{table}

\section{Conclusion}
\label{sec:conclusion}

We present a randomized algorithm for enumerating the vertices of a zonotope, which is a low-dimensional linear projection of a hypercube. The algorithm relies on a characterization of a zonotope vertex as a linear combination of the zonotope generators, where the weights of the combination are the signs of the matrix-vector product between the zonotope's generator matrix and a vector in $\R^n$. We study the probability of recovering particular vertices and relate it to the vertices' normal cones. This study shows that if we terminate the randomized algorithm before all vertices are recovered, then the convex hull of the resulting vertex set approximates the zonotope. 

The numerical examples suggest that the time needed to recover all vertices increases rapidly with the dimension $n$---since incrementing $n$ exponentially increases the size of the search space for the randomized algorithm. Therefore, for large $n$, we expect the enumeration algorithm to be most appropriate as an approximation algorithm---particularly for cases when existing methods (e.g., a reverse search-based method) are not practical. The algorithm may also be useful for generating a set of starting vertices for independent, parallel reverse searches. 

We hypothesize that one can develop a \emph{condition number} for the vertex enumeration problem based on how close two generating line segments are to parallel. If two line segments are close to parallel, then there will be a pair of vertices whose normal cones have small probability. One may be able to measure this condition number from the zonotope's generators, thus gaining insight into how well the randomized algorithm will perform before applying it. We leave this question to future work. 

\appendix

\section{Preliminary definitions}
\label{app:pre}
For a polytope $P \subset \R^n$ and $\vc\in \R^n$, $S(P,\vc)$ denotes the set of maximizers for the linear function $\ip{\vx}{\vc}$, i.e., 
\begin{equation}
S(P,\vc) \;=\; \{\,\vp\in P \,\mid\, \ip{\vp}{\vc} \geq \ip{\vq}{\vc}  \mbox{ for all $\vq\in P$}\,\}.
\end{equation}
Given polytopes $P_1,\dots,P_k$, $P_1+\cdots+P_k$ denotes their Minkowski sum. For $P = P_1+\cdots+P_k$, we call $P_1+\cdots+P_k$ the \emph{Minkowski decomposition} of $P$. For example, \eqref{eq:zonomink} is the Minkowski decomposition of the zonotope $Z(\mA)$.

\section{Proofs of Theorem \ref{thm:zono}, Corollary \ref{cor:vert-}, and Corollary \ref{cor:mapping}}
\label{app:0}

The key to establishing Theorem \ref{thm:zono} is a necessary and sufficient characterization of an extreme point of a polytope given the polytope's Minkowski decomposition from Fukuda~\cite{Fukuda2004}. We restate Fukuda's result for completeness.

\begin{theorem}[Corollary 2.2~\cite{Fukuda2004}]
\label{thm:fukuda}
Let $P_1,\dots,P_k$ be polytopes in $\R^n$ and let $P = P_1+\cdots+P_k$. A vector $\vv\in P$ is an extreme point of $P$ if and only if 
\begin{enumerate}
\item[(i)] $\vv=\vv_1+\cdots+\vv_k$, where $\vv_i$ is an extreme point of $P_i$, and 
\item[(ii)] there exists $\vc \in \R^n$ with $\{\vv_i\} = S(P_i,\vc)$ for all $i$.
\end{enumerate}
\end{theorem}

For the reader's convenience, we restate Theorem \ref{thm:zono} and Corollaries \ref{cor:vert-} and \ref{cor:mapping}.

\begin{customthm}{2} 
\label{thm:zono-app} 
Under Assumption~\ref{ass:big}, for $\vx \in \R^n$ such that $\mA^T \vx$ has all nonzero components, the point $\vv$ defined as
\begin{equation}
\label{eq:vert-app}
\vv \;=\; \vm(\vx) \;:=\; \mA\sign{\mA^T\vx}
\end{equation}
is a vertex of the zonotope $Z(\mA)$. 
\end{customthm}

\begin{proof}
Let $\vx\in\R^n$ satisfy the hypothesis of the theorem, and set $\vk=\sign{\mA^T\vx}$. Let $k_i$ denote the $i$th component of $\vk$, then $k_i = \sign{\ip{\va_i}{\vx}}$. If $k_i=1$, then $\ip{\va_i}{\vx}>0$. Recall the definition of $A_i$ from \eqref{eq:lineseg}, and note that $\ip{\va_i}{\vx}>0$ implies $S(A_i,\vx)=\{\va_i\}$. Similarly, if $k_i=-1$, then $S(A_i,\vx)=\{-\va_i\}$. By assumption, $k_i\neq 0$ for all $i$. Thus, $S(A_i,\vx)=\{k_i\va_i\}$ for all $i$. By Theorem \ref{thm:fukuda}, 
\begin{equation}
k_1\va_1 + \cdots + k_m\va_m \;=\; \mA\, \vk \;=\; \vm(\vx) \;=\; \vv
\end{equation}
is a vertex of the zonotope $A_1+\cdots+A_m=Z(\mA)$.
\end{proof}

\begin{customcor}{1}
\label{cor:vert-app}
If $\vv$ is a vertex of the zonotope $Z(\mA)$, then so is $-\vv$.
\end{customcor}

\begin{proof}
This follows immediately from Fukuda's result, Theorem \ref{thm:fukuda}. 
\end{proof}

\begin{customcor}{2} 
\label{cor:mapping-app}
Under Assumption~\ref{ass:big} and for $\vm(\vx)$ as in \eqref{eq:vert}, define $H\subset\R^n$ as
\begin{equation}
\label{eq:H-app}
H \;=\; \bigcup_{i=1}^m \{\, \vx\in \R^n\,\mid\, \ip{\va_i}{\vx} = 0 \,\}.
\end{equation}
The mapping $\vm:\R^n \setminus H\to \vertices{Z(\mA)}$ is well defined and onto.
\end{customcor}

\begin{proof}
By Theorem \ref{thm:zono}, the mapping is well defined. It remains to show $\vm$ is onto $\vertices{Z}$. By Theorem \ref{thm:fukuda} and $Z(\mA)$'s Minkowski decomposition \eqref{eq:zonomink}, for any $\vv\in \vertices{Z}$, 
\begin{equation}
\vv \;=\; \vv_1 + \cdots + \vv_m, 
\end{equation}
where $\{\vv_i\} = S(A_i,\vc)$ for $i=1,\dots,m$ and some $\vc \in \R^n$. For each $i$, we can write $\vv_i = k_i\va_i$, where $k_i=-1$ or $k_i=1$. Since $\{\vv_i\} = S(A_i,\vc)$, $\ip{k_i\va_i}{\vc}>0$, which implies $k_i = \sign{\ip{\va_i}{\vc}}$ and $\ip{\va_i}{\vc}\neq 0$ for all $i$. Thus, $\vc \in \R^n\setminus H$ and $\vm(\vc)=\vv$.
\end{proof}

\section{Proof of Theorem \ref{thm:normalCone}}
\label{app:A}

To prove Theorem \ref{thm:normalCone}, we need the following lemma, which characterizes the interior of the normal cone.

\begin{lemma} 
\label{lem:interiorNormal}
The interior of the normal cone $N_Z(\vv)$ can be written
\begin{equation}
\label{eq:S}
\interior{N_Z(\vv)} \;=\; S(\vv) \;:=\; \{\,
\vx\in\R^n \,\mid\, \ip{\vx}{\vz-\vv}< 0 
\mbox{ for all $\vz\in Z\setminus\{\vv\}$} 
\,\}.
\end{equation}
\end{lemma}

\begin{proof}
Note that
\begin{equation} 
\label{eqn:lem1}
N_{Z}(\vv) \;=\; \bigcap_{\vx \in Z\setminus \{\vv\}} \{\,
\vp\in\R^n \,\mid\, \ip{\vp}{\vx-\vv} \leq 0\,\}.
\end{equation}
This intersection ranges over uncountably many sets, and thus it is unclear if the interior commutes with intersection. However, we may show that 
\begin{equation} 
\label{eqn:finint}
N_{Z}(\vv) \;=\; \bigcap_{\vx \in \vertices{Z}\setminus \{\vv\}} \{\,
\vp\in\R^n \,\mid\, \ip{\vp}{\vx-\vv} \leq 0\,\}.
\end{equation}
Since the interior is commutative under finite intersection, 
\begin{equation} 
\label{eqn:lem2}
\begin{aligned}
\interior{N_{Z}(\vv)} 
&= 
\bigcap_{\vx \in \vertices{Z} \setminus \{\vv\}} \interior{\{\,\vp\in\R^n \,\mid\, \ip{\vp}{\vx-\vv} \leq 0\,\}} \\
&= 
\bigcap_{\vx \in \vertices{Z} \setminus \{\vv\}} \{\,\vp\in\R^n \,\mid\, \ip{\vp}{\vx-\vv} < 0\,\} \\
&= S(\vv),
\end{aligned}
\end{equation}
where (i) the second equality follows from continuity of the inner product and (ii) the last equality follows from \eqref{eq:S} and an equation analogous to (\ref{eqn:finint}). Thus, it suffices to show \eqref{eqn:finint}. By \eqref{eqn:lem1},
\begin{equation} 
\label{eqn:subineq}
N_{Z}(\vv) \;\subset\; 
\bigcap_{\vx \in \vertices{Z} \setminus \{\vv\}} \{\,\vp\in\R^n \,\mid\, \ip{\vp}{\vx-\vv} \leq 0\,\}.
\end{equation}
Choose $\vt$ and $\vx$ such that 
\begin{equation}
\vt \in \bigcap_{\vx \in \vertices{Z} \setminus \{\vv\}} \{\,\vp\in\R^n \,\mid\, \ip{\vp}{\vx-\vv} \leq 0\,\}, \qquad
\vx \in Z \setminus \{\vv\}.
\end{equation}
Since $Z$ is the convex hull of its vertices,
\begin{equation}
\vx \;=\; \sum_{\vv_i \in \vertices{Z}} \gamma_i \vv_i,
\end{equation}
where $\gamma_i\geq 0$ and $\sum_i\gamma_i=1$. Thus,
\begin{equation}
\ip{\vt}{\vx-\vv} 
\;=\; 
\left\langle 
\vt,\;\sum_{\vv_i \in \vertices{Z}} \gamma_i (\vv_i -\vv) \right\rangle 
\;=\; 
\sum_{\vv_i \in \vertices{Z}} \gamma_i\ip{\vt}{\vv_i -\vv} 
\;\leq\; 0,
\end{equation}
which implies
\begin{equation}\label{eqn:subineq2}
N_{Z}(\vv) \;\supset\; 
\bigcap_{\vx \in \vertices{Z} \setminus \{\vv\}} \{\,\vp \in\R^n \,\mid\, \ip{\vp}{\vx-\vv} \leq 0\,\}.
\end{equation}
Equation \eqref{eqn:finint} follows from \eqref{eqn:subineq} and \eqref{eqn:subineq2}.

\end{proof}

For the reader's convenience, we restate Theorem \ref{thm:normalCone}.

\begin{customthm}{3} 
\label{thm:normalCone-app}
Under Assumption~\ref{ass:big}, for $\vv\in\vertices{Z}$ and the mapping $\vm(\vx)$ from \eqref{eq:vert}, 
\begin{equation}
\vm^{-1}(\vv) \;=\; \interior{N_Z(\vv)}.
\end{equation}
\end{customthm}

\begin{proof}
By Lemma \ref{lem:interiorNormal}, it is enough to show that $S(\vv) = \vm^{-1}(\vv)$, where $S(\vv)$ is from \eqref{eq:S}. By Theorem \ref{thm:fukuda}, we can write 
\begin{equation}
\vv \;=\; \sum_{i=1}^{m} k_i\va_i,
\end{equation} 
where $k_i = 1$ or $-1$. Let $\vp \in S(\vv)$, fix $i$, and define
\begin{equation}
\vx \;=\; \sum_{\substack{j=1\\j \neq i}}^m k_j \va_j. 
\end{equation}
Then $\ip{\vp}{-k_i\va_i}=\ip{\vp}{\vx-\vv} <0$, by assumption. Thus, $k_i\ip{\vp}{\va_i} > 0$, and $\sign{\ip{\va_i}{\vp}} = k_i$. Therefore, $\vm(\vp) = \vv$. 

Suppose $\vp \in \vm^{-1}(\vv)$, and let $k_i = \sign{\ip{\va_i}{\vp}}$. Then $k_i\,\ip{\vp}{\va_i} > 0$, and for $\gamma < 0$, $\gamma\,k_i\,\ip{\vp}{\va_i} < 0$. Define $\gamma_i$ such that $\gamma_i\leq 0$ for all $i$ and $\gamma_i<0$ for some $i$. By linearity of the inner product, 
\begin{equation} \label{eqn:inequ}
\left\langle \vp,\;\sum_{i=1}^{m} \gamma_i \,k_i\,\va_i\right\rangle < 0.
\end{equation}
Let $\vx\in Z\setminus\{\vv\}$. Then 
\begin{equation}
\vx-\vv \;=\; \sum_{i=1}^{m} t_i\va_i - \sum_{i=1}^{m} k_i\va_i \;=\; \sum_{i=1}^{m} (t_i-k_i) \va_i,
\end{equation}
where $t_i \in [-1,1]$ for all $i$. If $k_i = 1$, then $t_i-k_i \leq 0$ and $t_i - k_i  = \gamma_i k_i$ for some $\gamma_i \leq 0$. If $k_i = -1$, then $t_i - k_i \geq 0$ and $t_i-k_i = \gamma_i k_i$, for some $\gamma_i \leq 0$. As $\vx \neq \vv$, $\gamma_i \neq 0$ for all $i$. Thus,
\begin{equation}
\vx-\vv \;=\; \sum_{i=1}^{m} \gamma_i\, k_i\,\va_i,
\end{equation}
where $\gamma_i \leq 0$ for all $i$ and $\gamma_i < 0$ for some $i$. By (\ref{eqn:inequ}), $\ip{\vp}{\vx-\vv} < 0$. As $\vx$ was arbitrary, $\vp \in S(\vv)$.
\end{proof}

\section{Proof of Theorem \ref{thm:bounds}}
\label{app:B}

For the reader's convenience, we restate Theorem \ref{thm:bounds}.

\begin{customthm}{4}
\label{thm:bounds-app}
Given $\varepsilon>0$ and $\delta>0$, let $K$ be the convex hull of a centrally symmetric set of points, and let $\vx_1,\dots,\vx_p$ be independent standard Gaussian vectors in $\mathbb{R}^n$. Define the subset $U_{K}\subseteq\vertices{K}$ dependent on $\delta$ as
\begin{equation}
U_{ K} \;=\; \{\, \vv\in\vertices{K} \,\mid\, \alpha_K(\vv)\geq \delta \,\},
\end{equation}
and define the event $A_K$ dependent on $\{\vx_i\}$ as
\begin{equation}
\label{eq:rv-app}
\{\vx_i\} \cap (N_K(\vv) \cup N_K(-\vv)) \not= \varnothing \mbox{ for all $\vv\in U_K$}.
\end{equation}
If $p$ is such that
\begin{equation}
p \;>\; \frac{
\log(|\vertices{K}|/\varepsilon)
}{
\log(1/(1-k))
},
\end{equation}
where
\begin{equation}
k \;=\; \left(
\frac{1}{2}\,(1-\sin(\arctan(b/\delta)))
\right)^{\frac{n-1}{2}},
\end{equation}
and
\begin{equation}
b \;\geq\; \underset{\vv\in\vertices{K}}{\max}
\diam{\base_K(\vv)},
\end{equation}
then $\mathbb{P}[A_K] \geq 1 - \varepsilon$.
\end{customthm}

\begin{proof}
Let $\vv\in\vertices{K}$. By Lemma 2.5 of Damle and Sun~\cite{damle2014random},
\begin{equation}
\alpha_K(\vv) 
\;\leq\; 
\frac{
\diam{\base_{K}(\vv)}
}{
\tan(\arcsin(1-\frac{1}{2}\,r(\omega_\vv)^2))
},
\end{equation}
where $\omega_\vv = \Px{N_{K}(\vv)}$ and $r(\omega_\vv) = 2(2\omega_\vv)^{\frac{1}{n-1}}$. This inequality holds for an upper bound $b$ of the set of diameters $\{\diam{\base_{K}(\vv)}\}$ with $\vv\in\vertices{K}$. Then, for $\vv\in\vertices{K}$ such that 
\begin{equation}
\omega_\vv \;\leq\;
\big(2[1-\sin(\arctan(b/\delta))]\big)^{\frac{n-1}{2}}/2^{n} \;=\; k,
\end{equation} 
we have that $\alpha_{K}(\vv) \leq \delta$.

For $\vv\in\vertices{K}$, consider the event that $\{\vx_i\}\cap (N_{K}(\vv)\cup N_K(\vv)) = \emptyset$; call this event $\{\text{miss } \vv\}$. Note that $\mathbb{P}[\{\text{miss } \vv \text{ with } \omega_\vv \geq k\}] \leq (1- 2k)^p$. Thus,
\begin{equation}
\mathbb{P}[ A_K^C ]
\;\leq\;
\mathbb{P}\left[
\bigcup_{\substack{\vv \in \vertices{K}\\ \omega_\vv \geq k}}\{\text{miss } \vv \}
\right] 
\;\leq\; 
\sum_{\substack{\vv \in \vertices{K}\\ \omega_\vv \geq k}} \mathbb{P}[\{\text{miss } \vv \}] 
\;\leq\; |\vertices{K}|\, (1- 2k)^p.
\end{equation}
Choosing $p$ such that $|\vertices{K}| (1- 2k)^p \leq \varepsilon$ or equivalently $p > \log(\frac{|\vertices{K}|}{\varepsilon}) / \log(\frac{1}{1-2k})$, we have $\mathbb{P}[A_K] \geq 1-\varepsilon$.
\end{proof}

\section{Proof of Theorem \ref{thm:hdist}}
\label{app:C}



For the reader's convenience, we restate Theorem \ref{thm:hdist}.

\begin{customthm}{5}
\label{thm:hdist-app}
Let $Z=Z(\mA)$ be a zonotope with generator $\mA\in\mathbb{R}^{n\times m}$ satisfying Assumption \ref{ass:big}. Given $\varepsilon>0$ and $\delta>0$, choose $p$ as in Theorem \ref{thm:bounds} for $b\geq\operatorname{diam}(Z)$, and let $V$ be the subset of $Z$'s vertices produced by Algorithm \ref{alg:vertenum} after $p$ iterations. Then
\begin{equation}
h(Z,\conv{V}) \;\leq\; \frac{|\vertices{Z}\setminus V|}{2}\, \delta
\end{equation} 
with probability at least $1-2^a\,\epsilon$, where $a=|\vertices{Z}\setminus U_{Z}|/2$ and 
\begin{equation}
U_{Z} \;=\; \{\,\vv \in \vertices{Z} \,\mid\, \alpha_Z(\vv) \geq \delta \,\}.
\end{equation}
\end{customthm}

\begin{proof}
Let $K_1$ and $K_2$ be convex hulls of two centrally symmetric subsets of $\vertices{Z}$, and define $A_{K_1}$ and $A_{K_2}$ as in \eqref{eq:rv}. By the law of total probability,
\begin{equation}
\Prob{A_{K_2}} \;=\; \Prob{A_{K_2}|A_{K_1}}\Prob{A_{K_1}}
+ \Prob{A_{K_2}|A_{K_1}^C}\Prob{A_{K_1}^C}.
\end{equation}
Therefore, applying Theorem \ref{thm:bounds},
\begin{equation}
\Prob{A_{K_2}|A_{K_1}}
\;=\; \frac{\Prob{A_{K_2}} - \Prob{A_{K_2}|A_{K_1}^C}\Prob{A_{K_1}^C}}{\Prob{A_{K_1}}} 
\;\geq\; \frac{1-2\varepsilon}{\Prob{A_{K_1}}}.
\end{equation}
Thus,
\begin{align}
\Prob{A_{K_1} \cap A_{K_2}}
&= \Prob{A_{K_2}|A_{K_1}}\Prob{A_{K_1}} \\
&\geq \frac{1-2\varepsilon}{\Prob{A_{K_1}}}\,\Prob{A_{K_1}} \\
&= 1-2\varepsilon
\end{align}
We can repeat this argument for a set $K_1,\dots,K_k$ of convex hulls of symmetric subsets of $\vertices{Z}$ to get
\begin{equation}
\Prob{\bigcap_{i=1}^{k} A_{K_i}} \;\geq\; 1- k\varepsilon.
\end{equation}
For notational convenience, we define $\pm N_K(\vv) = N_{K}(\vv) \cup N_K(-\vv)$. Let $\{K_i\}$ be the set of all convex hulls of symmetric subsets of $\vertices{Z}$ containing $U_{Z}$. In particular, $U_K\subset\vertices{K_i}$ for $i=1,\dots,|\{K_i\}|$, and  
\begin{equation}
|\{K_i\}| \;=\; 2^{\frac{|\vertices{Z}\setminus U_{K}|}{2}}.
\end{equation}
Consider a realization of $\vx_1,\dots,\vx_p$ such that the event $A_{K_i}$ holds for each $K_i$. Let $V \subset \vertices{Z}$ be the vertices such that $\{\vx_j\} \cap \pm N_{Z}(\vv) \neq \emptyset$ (i.e., the returned vertices), and define $n_V=|\vertices{Z}\setminus V|$. Order the set $\vertices{Z}\setminus V$ as $\{\,\vv_j \,\mid\, j=1,\dots,n_V\,\}$ such that $\vv_{j+1} = -\vv_j$ for odd $j$. Define 
\begin{equation}
C_i = \conv{\vertices{Z}\setminus \{\,\vv_j \,\mid\, j=1,\dots,2i\,\} },
\qquad
i = 0,\dots, n_V/2.
\end{equation} 
Note that
\begin{equation}
C_0 = Z, \qquad
C_{n_V/2} = \conv{V}.
\end{equation}
It is clear that $\{C_i\} \subset \{K_i\}$. Thus, $\{\vx_i\}$ satisfies $A_{C_i}$ for each $C_i$. Next we show that, for all $i$, 
\begin{equation}
\label{eqn:induction}
\alpha_{C_i}(\vv) < \delta
\mbox{ for all $\vv \in \{\vertices{Z}\setminus V\} \cap C_i$}.
\end{equation}
We show this first for $i=1$. Suppose by contradiction that for some $\vv \in \{\vertices{Z}\setminus V\} \cap C_1$, $\alpha_{C_1}(\vv) \geq \delta$. Then $\{\vx_j\} \cap \pm N_{C_1}(\vv)  \neq \emptyset$. But
\begin{equation}
\pm N_{C_1}(\vv) \subset
\big(
\pm N_{C_0}(\vv) \cup
\pm N_{C_0}(\vv_1)
\big).
\end{equation}
This follows since the normal cone is a subset of the dual space for which a vertex is the maximizing element of the convex set. Thus,
\begin{equation}
\{\vx_j\} \cap \big(
\pm N_{C_0}(\vv) \cup
\pm N_{C_0}(\vv_1)
\big)
\;\neq\; \varnothing,
\end{equation}
which is a contradiction as $C_0 =Z$. By a similar argument, we can show \eqref{eqn:induction} holds for any $C_i$ as
\begin{equation}
\begin{aligned}
\pm N_{C_i}(\vv)
&\subset
\pm N_{C_{i-1}}(\vv)
\cup
\pm N_{C_{i-1}}(\vv_i)\\
&\subset
\pm N_{C_0}(\vv) \cup \bigcup_{j=1}^{i} \pm N_{C_0}(\vv_j).
\end{aligned}
\end{equation}
By the definition \eqref{eq:simpconst}, \eqref{eqn:induction} shows that $h(C_i,C_{i+1}) \leq \delta$ for all $i$. Thus, 
\begin{equation}
\begin{aligned}
h(Z,\conv{V}) &= h(C_0, C_{n_V/2})\\
&\leq \sum_{i=0}^{n_V-1} h(C_i,C_{i+1})\\
&\leq n_V\, \delta. 
\end{aligned}
\end{equation}
Noting that
\begin{equation}
\Prob{\bigcap_{i=1}^{|\{K_j\}|} A_{K_i}} \;\geq\;
1 - 2^{\frac{|\vertices{Z}\setminus U_{Z}|}{2}}\,\epsilon
\end{equation}
completes the proof.
\end{proof}

\begin{acknowledgements}
The authors thank Al Charlesworth from Colorado School of Mines for helpful discussions. The first and third authors' work is supported by the U.S. Department of Energy Office of Science, Office of Advanced Scientific Computing Research, Applied Mathematics program under Award Number DE-SC-0011077. The second author's work is supported by NSF awards CCF-1149756 and CCF-093937. 
\end{acknowledgements}

\bibliographystyle{spmpsci}      
\bibliography{zonotopes}   

%
%

\end{document}